\documentclass[11pt,reqno]{amsart}
\usepackage[utf8]{inputenc}
\usepackage[english]{babel}
\usepackage{amsfonts,amssymb,amsmath,mathtools,braket,amssymb,dsfont,empheq}
\usepackage{enumerate,enumitem}
\usepackage{manfnt}
\usepackage[table, dvipsnames]{xcolor}
\usepackage{ amssymb }
\usepackage{todonotes}

\theoremstyle{plain}
\newtheorem{theorem}{Theorem}[section]
\newtheorem{lemma}[theorem]{Lemma}
\newtheorem{corollary}[theorem]{Corollary}
\newtheorem{proposition}[theorem]{Proposition}
\newtheorem{definition}{Definition}[section]
\newtheorem*{namedthm}{\namedthmname}

\newcounter{namedthm}
\makeatletter
\newenvironment{named}[1]
  {\def\namedthmname{#1}%
   \refstepcounter{namedthm}%
   \namedthm\def\@currentlabel{#1}}
  {\endnamedthm}

\theoremstyle{definition}
\newtheorem{remark}[theorem]{Remark}

\newcommand{\tr}{\operatorname{tr}}
\renewcommand{\leq}{\leqslant}
\renewcommand{\geq}{\geqslant}

\newcommand{\N}{\mathbb{N}} 
\newcommand{\R}{\mathbb{R}} 

\newcommand{\bbS}{\mathbb{S}}

\newcommand{\cB}{\mathcal{B}}
\newcommand{\cC}{\mathcal{C}}

\newcommand{\cL}{\mathcal{L}}
\newcommand{\cM}{\mathcal{M}}
\newcommand{\cO}{\mathcal{O}}
\newcommand{\cR}{\mathcal{R}}
\newcommand{\cT}{\mathcal{T}}

\newcommand{\norm}[1]{\left\vert\kern-0.25ex\left\vert #1 \right\vert\kern-0.25ex\right\vert}
\newcommand{\normt}[1]{\left\vert\kern-0.25ex\left\vert\kern-0.25ex\left\vert #1 \right\vert\kern-0.25ex\right\vert\kern-0.25ex\right\vert}
\newcommand{\inprod}[2]{\left\langle #1, #2 \right \rangle}

\newcommand{\di}{\,\mathrm{d}}

\DeclareMathOperator{\dist}{dist}
\DeclareMathOperator{\Ker}{Ker}

\newcommand{\Crit}{\mathcal{C}\mathcal{M}\mathcal{C}}

\usepackage{float}
\usepackage{graphicx}
\usepackage{subcaption}
		
\usepackage{hyperref}

\title[Quantitative Stability for Yamabe minimizers with boundary]{Quantitative Stability for Yamabe minimizers on manifolds with boundary}

\author[B. Borquez]{Benjam\'in Borquez}
\address{Departamento de Ingenier\'ia Matem\'atica, Beauchef 851, Santiago, Chile}
\email{benjamin.borquez@ug.uchile.cl}

\author[R. Caju]{Rayssa Caju}
\address{Departamento de Ingenier\'ia Matem\'atica and Centro
de Modelamiento Matem\'atico (CNRS IRL 2807), Universidad de Chile, Beauchef 851, Santiago, Chile}
\email{rcaju@dim.uchile.cl}

\author[H. Van Den Bosch]{Hanne Van Den Bosch}
\address{Departamento de Ingenier\'ia Matem\'atica and Centro
de Modelamiento Matem\'atico (CNRS IRL 2807), Universidad de Chile, Beauchef 851, Santiago, Chile}
\email{hvdbosch@dim.uchile.cl}

\date{\today}

\begin{document}

\begin{abstract}
This paper addresses the quantitative stability for a Yamabe-type functional on compact manifolds with boundary introduced by Escobar. Minimizers of the functional correspond to scalar-flat metrics with constant mean curvature on the boundary.
We prove that the deficit controls the distance to the minimizing set to a suitable power by reducing the problem to the analogous question for an effective functional on the boundary.
\end{abstract}

\maketitle

\section{Introduction}
\noindent

The classical Yamabe problem asks whether, given a closed Riemannian manifold $(M,g)$, it is possible to find a conformal metric $\tilde{g}$ to $g$, such that the scalar curvature of $\tilde{g}$ is constant. Writing the conformal factor as $\tilde{g} = u^{\frac{4}{n-2}}g$, this problem is equivalent to finding a positive solution $u >0$ to the equation
\begin{equation*}
      \frac{4(n-1)}{n-2} \Delta_{g} u - R_{g} u + n(n-1)u^{\frac{n+2}{n-2}} = 0, 
\end{equation*}
where $\Delta_{g}$ denotes the Laplace-Beltrami operator of the metric $g$ and $R_{g}$ is the scalar curvature.
The Yamabe problem was solved by the combined efforts of the works of Yamabe \cite{yamabe1960deformation}, Trudinger \cite{trudinger1968remarks}, Aubin\cite{aubin1976equations} and Schoen \cite{schoen1984conformal}. Solutions to this problem are critical points of the Yamabe quotient
\begin{equation}\label{YQ}
    Q(u) = \frac{\int_{M}(c_{n} |\nabla u|^{2} + R_{g} u^{2})d\text{vol}_{g}}{(\int_{M}u^{2^{*}}d\text{vol}_{g})^{2/2^{*}}}.
\end{equation}
where $c_{n} = 4(n-1)/n-2$. A conformally invariant quantity that plays a key role in the resolution of this problem is known as the Yamabe invariant, and it is defined as
\begin{equation*}
    Y(M,[g]) := \inf \{ Q(u) : u \in H^{1}(M,\mathbb{R}_{+})\}.
\end{equation*}

Once the properties of the existence of minimizers and their structure are understood, a natural question is the stability of the inequality: \emph{if $Q(u)$ is close to the infimum $Y(M,[g])$, does this imply that $u$ is close to a minimizer}?

For the case of the round sphere, the answer is afirmative.
By combining the classical Bianchi-Egnell's inequality \cite{BianchiEgnell}, 
with the conformal equivalence between the Euclidean space and the round sphere, it is possible to bound from below the deficit between the Yamabe quotient and its minimum $Q_{(\mathbb{S}^{n},g_{0})}(u) - Y(\mathbb{S}^{n},g_{0})$, for a given function $u \in W^{1,2}(\mathbb{S}^{n})$, in terms of the square of a distance to the family of minimizers $\mathcal{M}(\mathbb{S}^{n},g_0)$ of $\eqref{YQ}$. In the case of the round sphere, this class of minimizers is explicitly characterized by the works of Aubin \cite{aubin1976equations} and Talenti \cite{talenti1976best}.

In an elegant argument combining a Lyapunov-Schmidt reduction with the finite-dimensional version of the \L ojasiewicz inequality, Engelstein, Neumayer, and Spolaor obtained in \cite{engelstein2022quantitativestabilityminimizingyamabe} a generalization of the Bianchi-Egnell's result for any closed Riemannian manifold $(M,g)$ with dimension $n\geq 3$. They established the following quantitative stability result for the Yamabe minimizers.
\begin{named}{Theorem A}\cite{engelstein2022quantitativestabilityminimizingyamabe}\label{robin}
   Let $(M,g)$ be a closed Riemannian manifold of dimension $n\geq 3$ that is not conformally equivalent to the round sphere. Then there exists constants $
   C>0$ and $\gamma \geq 0$, depending on $(M,g)$, such that:
\begin{equation*}
    Q(u)-Y(M,[g]) \geq 
    C d(u,\mathcal{M})^{2+\gamma} \quad \forall u \in H^{1}(M,\mathbb{R}_{+}),
\end{equation*}
Where $d$ is a suitable bounded distance to the set of minimizers 
$\cM$.
Moreover, there exists an open dense subset $\mathcal{G}$ in the $C^{2}$ topology on the space of smooth-conformal classes of metrics on $M$ such that if $[g] \in \mathcal{G}$, we may take $\gamma = 0$. 
\end{named}

Furthermore, as a consequence of their result, a conformally invariant stability inequality is established. While quadratic stability (i.e., the case where $\gamma = 0$) is generic, their result is, in fact, optimal, as the existence of manifolds with a strictly positive exponent $\gamma$ is proven.

Motivated by the question mentioned above, our purpose in this paper is to establish a trace-type quantitative stability result for a compact Riemannian manifold \emph{with boundary}.
More precisely, we aim to prove a quantitative stability result for the energy functional associated with the formulation for the Yamabe problem proposed in the seminal work of J. Escobar \cite{escobar1992conformal}. In this case, we look for a conformal metric which is scalar-flat and has a constant mean curvature on the boundary.

More precisely, let $(M,g)$ be a compact Riemannian manifold with boundary $\partial M$ of dimension $n \geq 3$. By writing the conformal factor $\tilde{g} = u^{\frac{4}{n-2}}g$, the problem proposed above is equivalent to finding a positive solution to the nonlinear boundary-value problem
\begin{equation}\label{YPB}
     \begin{cases}
        -\frac{4(n-1)}{n-2} \Delta _{g} u + R_{g} u = 0 & M \\
        \frac{4(n-1)}{n-2}\frac{\partial u}{\partial \nu} + 2(n-1) h_{g} u = 2(n-1)cu^{\frac{n}{n-2}} & \partial M
   \end{cases}
\end{equation}
where here, $h_{g}$ is the mean curvature of the boundary with respect to $g$, $\nu$ is the outward unit normal vector to $\partial M$, and $c$ is a constant. In fact, given a solution $u$ to problem $\eqref{YPB}$, the metric $\tilde{g} = u^{\frac{4}{n-2}}g$ is scalar-flat and its boundary mean curvature is equal to $c$. This is a variational problem, associated with the Escobar functional
\begin{equation}\label{EYPB}
    Q(u) = \frac{\int_{M}(c_{n} |\nabla u|^{2} + R_{g} u^{2})d\text{vol}_{g} + \int_{\partial M} 2(n-1)h_{g}u^{2}d\sigma_{g}}{\left(\int_{\partial M}|u|^{2\frac{(n-1)}{(n-2)}}d\sigma_{g}\right)^{(n-2)/(n-1)}}.
\end{equation}
The exponent $\frac{2(n-1)}{n-2}$ is critical for the Sobolev trace embedding $H^{1}(M) \hookrightarrow L^{\frac{2(n-1)}{n-2}}(\partial M)$.
In \cite{escobar1992conformal}, Escobar considered the \emph{Sobolev quotient} associated with this version of the Yamabe problem for manifolds with boundary, which is a conformally invariant quantity defined as
 \begin{equation} \label{eq:def_Q}
    Q(M,\partial M) := \inf \{Q(u): u \in C^{1}(\overline{M}), u > 0 \}.
\end{equation}
The sign of $Q(M,\partial M)$ is equal to the sign of the boundary mean curvature $c$.

Escobar showed that $Q(M,\partial M) \leq Q(B^{n},\partial B^{n})$, where $B^{n}$ denotes the unit ball in $\R^{n}$ endowed with the Euclidean metric. It was also proven that if \( Q(M, \partial M) \) is finite and the inequality is strict, then there exists a minimizing solution to the problem 
\eqref{YPB}. In this first paper, Escobar guaranteed existence of solutions to a large class of manifolds. Combinations of the works of Cherrier \cite{cherrier1984problemes}, Escobar \cite{escobar1992conformal,escobar1996conformal}, Marques \cite{marques2005existence,marques2007conformal}, Almaraz \cite{almaraz2010existence} and Chen \cite{chenconformal}, showed the existence of a minimizer for functional $Q$ for many types of manifolds. The remaining cases were covered by a work of Mayer and Ndiaye \cite{mayer2017barycenter}, using the barycenter technique of Bahri and Coron. 
Summing up, the strict inequality holds for any manifold with boundary in dimensions $n\ge 3$ with finite invariant $Q(M,\partial M)$ and which is not conformally equivalent to the ball.

To state our result, we fix a compact manifold with boundary $(M,\partial M)$ and a reference metric, and denote by $\cM$ the set of minimizers of \eqref{eq:def_Q}.
For any $u\in H^1(M)$, we define the distance to this set as
\begin{equation*}
    d(u, \mathcal{M}) := 
    \frac{\inf_{v \in \mathcal{M}}\|u - v\|_{H^{1}(M)}}{\|u\|_{H^{1}(M)}}.
\end{equation*}
Note that this distance is invariant under scaling of $u$, and bounded from above. 

\begin{theorem}\label{thm:mtheor}
    Let $(M,g)$ be a compact Riemannian manifold with boundary of dimension $n\geq 3$, such that $Q(M,\partial M) < Q(B^{n},\partial B^{n})$. Assume also that the reference metric $g$ has $R_{g}\geq 0$. Then there exist constants $C >0$ and $\gamma \geq 0$, depending on $(M,g)$, such that:
\begin{equation*}
    Q(u) - Q(M,\partial M) \geq C d(u,\mathcal{M})^{2+\gamma} \quad \text{ for all } u \in H^{1}(M,\mathbb{R}_{+}).
\end{equation*}
\end{theorem}
The stability estimate in Theorem \ref{thm:mtheor} provides a quantitative measure of how far a function is from the set of minimizers in a general Riemannian setting, for a manifold with boundary. 

In the Euclidean case, the first result establishing a sharp trace Sobolev inequality is due to Escobar \cite{EscobarSharpConstantTrace}, who proved that for a smooth real valued function $u$ defined in the upper half $n$-dimensional euclidean space $\mathbb{R}_{+}^{n}$, then
\begin{equation*}
    \left( \int_{\partial \mathbb{R}_{+}^{n}} |u|^{2\frac{n-1}{n-2}}   \right)^{\frac{n-2}{n-1}} \leq K^{-1} \int_{\mathbb{R}_{+}^{n}} |\nabla u|^{2} 
\end{equation*}
where $K = \left(\frac{n-2}{2}\right) (\sigma_{n-1})^{1/(n-1)}$, and $\sigma_{n-1}$ is the volume of the $(n-1)$-dimensional unit sphere. Moreover, the extremal functions of the trace Sobolev inequality were fully characterized. These functions play the role of the Talenti functions in the case without boundary. 

A natural question that arises is how to extend the Bianchi-Egnell inequality \cite{BianchiEgnell} to the trace setting. This was addressed in the work of Ho in \cite{Ho}, who proved that there exists a positive constant $\alpha$, depending only on the dimension, such that:
\begin{equation} \label{eq:Ho-result-flat-space}
    \|\nabla u \|_{L^{2}(\mathbb{R}_{+}^{n})}^{2} - S^{2}\|u\|_{L^{\frac{2(n-1)}{(n-2)}}(\partial \mathbb{R}_{+}^{n})}^{2} \geq \alpha d_{\dot H^1}(u,\mathcal{M}_{\R^n_+})^2 \quad \forall u \in H_{0}^{1}(\mathbb{R}_{+}^{n})
\end{equation}
where
\begin{align*}
    d_{\dot H^1}(u,\mathcal{M}) := \inf_{v \in \mathcal{M}_{\mathbb{R}_{+}^{n}}} \|\nabla(u - v) \|_{L^{2}(\mathbb{R}_{+}^{n})}.
\end{align*}

As expected, using the conformal equivalence between the Euclidean ball and the upper half-space $\R^{n}_{+}$, along with the trace-type Bianchi-Egnell inequality obtained by Ho in \cite{Ho}, it is possible to estimate the gap of the difference between the Yamabe quotient and the Yamabe invariant of the unit ball. We describe this in detail in Section \ref{sectionball}.

\medskip
For this general case, the overall idea of our proof is to reduce the argument to the analysis of a boundary functional.
While working on functions defined in the fractional Sobolev space of the boundary, we apply a Lyapunov-Schmidt reduction to the restriction of the boundary Yamabe functional to the kernel of the second variation. This is combined with the finite-dimensional version of the Łojasiewicz inequality to obtain local quantitative estimates on the boundary. Using a suitable harmonic extension operator, we extend the result to the whole manifold. 

This step is where the nonnegativity of the scalar curvature $R_{g}$ is required to guarantee the existence of such an extension and the equivalence between some norms.

\begin{remark}
Similarly to the closed case, we expect the quadratic stability in Theorem \ref{thm:mtheor} to hold generically, i.e., that for an open dense subset in the \(C^{2}\) topology on the space of conformal classes of \(C^{\infty}\) metrics on \(M\), it is possible to take \(\gamma = 0\). However, the tools needed to follow either the strategy of \cite{andrade2024quantitative} or that of \cite{anderson2005uniqueness} are still lacking in the boundary trace setting.
\end{remark}

We are also able to obtain as a corollary of Theorem \eqref{thm:mtheor}, a conformally invariant stability estimate.  Define the following distance:
\begin{align*}
    d_{[g]}(g_{u},g_{v}) := \left(\int_{M} |u-v|^{2^{*}} d\text{vol}_{g}\right)^{1/2{*}}
\end{align*}
where \( g_u := u^{\frac{4}{n-2}} g \) and $2^{*} = 2n/n-2$ is the critical Sobolev exponent. Even tough $d_{[g]}$ is defined with respect to a fixed representative metric \( g \in [g] \) in the conformal class, it is independent of this choice. Similarly, when \( Q := Q(M, \partial M) \geq 0 \), we define
the conformally covariant operators
\begin{align}
    L_g &:= - c_n \Delta +R_g \,: \, H^1( M) \mapsto H^{-1}(M), \quad   \text{ and }\\
      B_g &:= - c_n \partial_{\nu_g} + 2(n-1) h_g \,: \, H^{1/2}(\partial M) \mapsto H^{-1/2}(\partial M)
\end{align}
with this, define
\begin{align} \label{eq:def_d_1}
   d_{[g]}^1(g_{u},g_{v}) := \left(\int_{M} (u-v) L_g (u-v)d\text{vol}_{g} + \int_{\partial M} (u-v)B_g (u-v) d\sigma_g\right)^{1/2}
\end{align}
for any $g \in \mathcal{M}(M,g)$ with $\operatorname{vol}_{g}(M)=1$. As before, this distance is independent of the choice of $g$.

The functional $Q$ defined in \eqref{EYPB} measures the scalar curvature plus total mean curvature normalized by the boundary area, so we can define the equivalent functional defined over the class of conformal metrics $\tilde{g} \in [g]$, given by
\begin{align*}
  \mathcal{E}(\tilde{g}) = \frac{\int_{M}R_{\tilde{g}}d\text{vol}_{\tilde{g}} + \int_{\partial M}2(n-1) h_{\tilde{g}}d\sigma_{\tilde{g}}}{\left(\int_{\partial M} d\sigma_{\tilde{g}}\right)^{\frac{n-2}{n-1}}}.
\end{align*}
See for instance \cite{araujo2003critical} for more details.
With this in place, we can state the following corollary:
\begin{corollary}\label{corollary}
    Let $(M,g)$ be a compact Riemannian manifold with boundary of dimension $n \geq 3$. There exists constants $c>0$ and $\gamma \geq 0$, depending on $M$ and $[g]$, such that:
    \begin{align*}
        \mathcal{E}_{g} - Q(M,\partial M) \geq c \left(\frac{\inf \{ d_{[g]}(g,\tilde{g}) : \tilde{g}\in \mathcal{M} \}}{\operatorname{vol}_{g}(\partial M)^{\frac{n-2}{n-1}}}\right)^{2+\gamma} 
    \end{align*}
When $Q = Q(M,\partial M) \geq 0$ and $\mathcal{E}_{g} - Q(M,\partial M) \leq 1$, there exist constants $c>0$ and $\gamma \geq 0$ depending on $M$ and $[g]$ such that:
\begin{align*}
    \mathcal{E}_{g} - Q(M,\partial M) \geq c \left(\frac{\inf \{d_{[g]}^1(g,\tilde{g}) : \tilde{g}\in \mathcal{M} \}}{\operatorname{vol}_{g}(\partial M)^{\frac{n-2}{n-1}}}\right)^{2+\gamma} 
\end{align*}

\end{corollary}
The proof of this corollary is direct from Theorem~\ref{thm:mtheor} and the definition of the quantities involved, see also \cite{engelstein2022quantitativestabilityminimizingyamabe} for details on the conformal invariance of the involved quantities.
\medskip

Analogously to \ref{robin}, we expect the quadratic stability, i.e., $\gamma = 0$, to hold generically.
Indeed, quadratic stability holds whenever the minimizers are non-degenerate (the linearized equation has trivial kernel) or integrable (see Definition~\ref{def:integrability}), which covers the cases of non-uniqueness associated to symmetries of the problem.
On the other hand, we expect that for particular manifolds, $\gamma > 0$ is necessary. This case can be detected by the so called \emph{Adam-Simons positivity condition},inspired by the works of \cite{adams1988rates} and defined in our setting by \cite{hoflow}. The Adams-Simon condition, is related to the first nontrivial term in the analytic expansion of the Lyapunov–Schmidt reduction of the Yamabe functional at the metric, with its order denoted by $p$. We prove the following

\begin{proposition}[$\text{AS}_{p}$ implies $\gamma >0$]\label{prop:ASpint}
    Fix a compact Riemannian manifold with boundary of dimension $n \geq 3$, $R_{g} \geq 0$, and $p \geq 3$. Let $u_{0}$ be a non-integrable critical point of the Yamabe energy and suppose that it satisfies the $\text{AS}_{p}$ condition. Then there exists a sequence $(u_{i})_{i} \subset H^{1}(M)$ with $u_{i}\to u_{0}$ in $H^{1}(M)$, such that:
    \begin{align*}
        \lim_{i \to \infty} \frac{Q(u_{i})-Q(M,\partial M)}{\|u_{i} - u_{0} \|_{H^{1}(M)}^{p-\alpha}} = 0 \quad \forall \alpha >0
    \end{align*}
\end{proposition}

From this proposition, it suffices to show explicit examples of compact manifolds with boundary $(M,g)$, where g is a minimizer of the Yamabe energy satisfying $AS_p$ condition for $p \geq 3$. For the closed case, Carlotto, Chodosh, and Rubinstein in \cite{Ottis} showed that the product metric on $\mathbb{S}^{1}(1/\sqrt{n-2})\times \mathbb{S}^{n-1}(1)$ is a
nonintegrable critical point satisfying the Adams-Simon positivity condition for some $p \ge 4$, providing an example of the $AS_4$-condition for critical points. In \cite{frank2022degenerate} it is then shown that for this case, the critical point corresponds to the minimizer and a stability estimate with $\gamma = 2$ holds.

We expect that analogous phenomena occur for the trace-functional studied here. In \cite{cardenas2022bifurcation} it is shown that non-integrable \emph{critical points} occur on product manifolds as the constant in the product metric varies. However, showing that these critical points correspond to minimizers seems challenging in our context since we are dealing with a boundary nonlinearity.

\medskip
Overall, Sobolev inequalities, their sharp constants, and the properties of minimizers and critical points have been a subject sdof interest for several decades. We refer the reader to the following lecture notes as a survey on the subject \cite{frank2024sharp}.
The result of Bianchi and Egnell \cite{BianchiEgnell} was the first to provide a quantitative characterization of the Sobolev inequalities discussed in Lions' work \cite{lions1985concentration}. In a recent paper, Dolbeault, Esteban, Figalli, Frank, and Loss \cite{dolbeault2022sharp} presented a new proof of Bianchi and Egnell's inequality and obtained a completely sharp result. Sobolev inequalities in general, and also in the quantitative sense were also studied in the Riemannian setting as in the works of \cite{chen2024sharp,li1997sharp,nobili2022rigidity,nobili2025quantitative,nobili2024stability} an\cite{nobili2025quantitative}. In other contexts, the general study of quantitative stability for functional and geometric inequalities has been extensively developed, particularly in relation to the isoperimetric inequality, see for instance the works \cite{ciraolo2017shape, krummel2017isoperimetry}. In fact, the idea of using the Łojasiewicz inequality to study quantitative stability was first introduced by Chodosh, Engelstein, and Spolaor in \cite{chodosh2022riemannian}.

This work is organized as follows: In Section 2, we provide details on how to use conformal deformations to move freely between the half-space and the unit ball. In Section 3, we state some technical lemmas concerning the second variation of the boundary Yamabe functional. In Section 4, we present the local and global quantitative results and, finally, discuss the sharpness of the exponent in our result, which is related to the Adams-Simon positivity condition.

\subsection{Acknowledgments}
We thank Jesse Ratzkin for helpful discussions about this manuscript. The second author was partially supported by Fondecyt grant \# 11230872. The first and third authors received support from Fondecyt grant \# 11220194. The second and third authors were also supported by Centro de Modelamiento Matemático (CMM) BASAL fund FB210005 for center of excellence from ANID-Chile.

\section{The case of the ball}\label{sectionball}
As we mentioned before, the model case for the scalar flat problem is the Euclidean ball. Here, we establish the best-case scenario of our main result: when $\gamma = 0$. The stability comes from some algebraic manipulations of Ho's result \cite{Ho}. The inverse of the conformal diffeomorphism between the upper Euclidean half-space and the Euclidean ball is given 
\begin{equation*}
    F^{-1}(x_{1},...,x_{n-1},t) = \left(\frac{4x}{(2+t)^{2} + |x|^{2}}, \frac{4-t^{2} - |x|^{2}}{(2+t)^{2} + |x|^{2}} \right)
\end{equation*}
where $x = (x_{1},...,x_{n-1})$. Therefore, if $ds^{2}$ is the euclidean metric in $\mathbb{R}_{+}^{n}$ and $dy^{2}$ is the euclidean metric in $B$ we have:
\begin{equation*}
    4((1+t)^{2} + |x|^{2})^{-2} ds^{2} = (F^{-1})^{*}(dy^{2})
\end{equation*}
which we can write as $4v^{4/(n-2)}ds^{2}$, where:
\begin{equation*}
    v(x,t) = ((1+t)^{2} + |x|^{2})^{(2-n)/2}
\end{equation*}
With that, for $\varphi \in C^{\infty}(\bar{B})$, we define $\bar{\varphi} \in C^{\infty}(\mathbb{R}_{+}^{n})$ as the weighted pushforward function in $\mathbb{R}_{+}^{n}$. Explicitly:
\begin{equation*}
    \bar{\varphi} := v (F^{-1})^{*}(\varphi) = v \cdot \varphi \circ F^{-1}
\end{equation*}
As  in \cite{Ho}, the following identities hold
\begin{align} \label{eq:conf_sphere_grad}
    \int_{\mathbb{R}_{+}^{n}} |\nabla \bar{\varphi}|^{2}dxdt &= \int_{B} |\nabla \varphi|^{2}dy + (n-2)\int_{\partial B} |\nabla \varphi|^{2}dS(y)\\
    \label{eq:conf_sphere_bdry}
    \int_{\partial \mathbb{R}_{+}^{n}} |\bar{\varphi}|^{2\frac{n-1}{n-2}}dx &= \int_{\partial B} | \varphi|^{2\frac{n-1}{n-2}}dS(y).
\end{align}

The set of optimizers for the halfspace $\cM_{\R^d_+}$ is explicitly known from to be
\[
\cM_{\R^d_+} = \{ c D_\lambda T_{x_0} v | c, \lambda > 0, x_0 \in \partial \R^d_+\}
\]
where $T_{x_0}$ is the translation by $x_0$, and $D_\lambda$ denotes dilations of $(x,t)$. 

From \cite{EscobarSharpConstantTrace}, we know that
\begin{align*}
    Q(B^{n},\partial B^{n}) = Q(\mathbb{R}_{+}^{n}, \partial \mathbb{R}_{+}^{n}),
\end{align*}
Thus, we obtain:
\begin{proposition}
    Let $\varphi \in H_{0}^{1}(B)$, then there exists $C>0 $ such that
\begin{align*}
        Q_{B}(\varphi) - Q(B^n, \partial B^n) \geq C  d_{B}(\bar\varphi, \mathcal{M}_{B^n})^{2}
\end{align*}
where
\begin{align*}
        d_{B^n}(\varphi, \mathcal{M}_{B}) := \frac{\inf_{u \in \mathcal{M}_{B}}\|\varphi - u\|_{H^{1}(B)}}{\|\varphi\|_{H^{1}(B)}}
\end{align*}
\end{proposition}
\begin{proof}
\eqref{eq:Ho-result-flat-space}, 
   For the left-hand-side, we use the conformal invariance to reduce it to the half-space problem by \eqref{eq:conf_sphere_grad} and \eqref{eq:conf_sphere_bdry}.
   Then, using \eqref{eq:Ho-result-flat-space} from \cite{Ho}, and the conformal invariance of $d_{[B_n]}^1$ defined in \eqref{eq:def_d_1}
\begin{align*}
        Q_{B^n}(\varphi) - Q(B^n,\partial B^n) 
        &\geq C \inf_{c, x_0,\lambda} \norm{\bar \varphi}_{L^p}^{-2} \inf_{\bar \psi \in \cM_{R^n_+}}^{-2} d_{\R^n_+}^1(\bar \varphi, \bar \psi) \\
        &= C \norm{\varphi}_{L^p(\partial B^n)}^{-2} \inf_{\psi \in \cM_{B^n}} \left( \norm{\nabla (\phi -\psi)}_{L^2(B^n)}^2 + h_{\partial B^n}\norm{\phi -\psi}_{L^2(\partial B^n)}^2 \right)
\end{align*}
   If $k_n>0$ is the first eigenvalue of the Robin Laplacian with parameter $\alpha =2(n-2)$ on the ball, we find that 
       \begin{align*}
        1/2 \int_{B^n}|\nabla \eta|^{2} dy + (n-2)\int_{\partial B^n}|\eta|^{2} dS(y) \geq k_n/2\|\eta \|_{L^{2}(B^n)}^{2}.
    \end{align*}
    On the other hand, 
    \begin{equation*}
        \norm{\varphi}_{L^p(\partial B^n)}^{-2} \le C \norm{\varphi}_{H^1(B^n)}^2,
    \end{equation*}
    so the right-hand side controls $d_B(\phi, \cM_{B^n})$.
   \end{proof}

\section{General manifolds: notation and preliminary constructions.}

In this part we will fix some $g \in [g_{0}]$ and it will be implicit in the definition of the Sobolev spaces involved. We recall the Yamabe energy for the scalar flat case:
\begin{equation*}
    Q(u) = \frac{\int_{M}(c_{n} |\nabla u|^{2} + R_{g} u^{2})d\text{vol}_{g} + \hat{c}_{n}\int_{\partial M} h_{g}u^{2}d\sigma_{g}}{\left(\int_{\partial M}|u|^{2\frac{(n-1)}{(n-2)}}d\sigma_{g}\right)^{(n-2)/(n-1)}}
\end{equation*}
where $\hat{c}_{n}= 2(n-1)$. In order to gain compactness of the set of minimizers, as usual in the literature, we work in the Banach manifold
\begin{align}
    \mathcal{B} &:= \left\{u\in H^{1/2}(\partial M,\mathbb{R}_{+}): \int_{\partial M} u^{\frac{2(n-1)}{n-2}} d\sigma_{g} = 1 \right\}.
\end{align}

\subsection{The boundary functional \texorpdfstring{$\tilde Q$}{Q~}.}
Because $\cB$ only takes into account the boundary trace of the functions, we will work with a related functional for nonnegative functions on the boundary only. This functional is a fractional analogue of the Yamabe quotient for manifolds without boundary. 
In order to do so, we define a harmonic extension operator $E: H^{1/2}(\partial M) \mapsto H^{1}(M)$ that maps a function $u$ on $\partial M$ to $Eu$, the solution of
\begin{align} \label{eq:def_harmonic_extension}
       \begin{cases}
        L_g (Eu) = 0 & \text{ in } M\\
        Eu= u & \text{ on }\partial M
    \end{cases}
   .
\end{align}

Since $M$ is compact, the trace is a linear continuous operator $\tr(.)$ so that $H^{1/2}(\partial M)\subseteq L^{2}(\partial M)$ is the image of $H^{1}(M)$ by $\tr(.)$. Note that $E$ is well-defined because we assume $R_g \ge 0$, such that $-c_n \Delta  + R_g$ is a positive elliptic operator. Furthermore, if $u>0$ on $\partial M$, the maximum principle shows that $E(u)>0$ in the interior.

Next, we define the boundary functional
$\tilde{Q}:H^{1/2}(\partial M) \mapsto \mathbb{R} $ as
\begin{align*}
    \tilde{Q}(u) := Q(Eu) = \frac{\int_{M}(c_{n} |\nabla E u|^{2} + R_{g} (Eu)^{2})d\text{vol}_{g} + \hat{c}_{n}\int_{\partial M} h_{g}u^{2}d\sigma_{g}}{\left(\int_{\partial M}|u|^{2\frac{(n-1)}{(n-2)}}d\sigma_{g}\right)^{(n-2)/(n-1)}}.
\end{align*}
Integrating by parts we obtain
\begin{align*}
    \tilde{Q}(u) = \frac{   \langle (\Lambda + \hat c_n h_g) u,u\rangle 
    }{\left(\int_{\partial M} |u|^{2\frac{(n-1)}{(n-2)}}d\sigma_{g}\right)^{(n-2)/(n-1)}},
\end{align*}
where $\Lambda:H^{1/2}(\partial M) \longrightarrow H^{-1/2}(\partial M)$ is the Dirichlet-to-Neumann map associated to the elliptic operator $-c_n \Delta_g + R_g$. This map sends a function $u$ defined on the boundary to $\partial_n E(u)$. It is continuous from $H^s(\partial M)$ to $H^{s-1}(\partial M)$ for all $s \in \R$. Schauder estimates (e.g \cite[Theorem 6.6]{gilbarg2001elliptic}) show that it is also bounded from $C^{2,\alpha}(\partial M)$ to $C^{1,\alpha}(\partial M)$. 
With a slight abuse of notation, we use $\inprod{\cdot}{\cdot}$ both for the $L^2$-inner product with respect to the surface measure on $\partial M$ or the associated coupling between $H^{-1/2}$ and $H^{1/2}$.

With this functional, we can define the conformal invariant in a manner that involves only functions on the boundary.
\begin{lemma} \label{lem:infequa}
    With the previous definitions, we have
    \begin{align*}
        Q(M,\partial M):= \inf \{Q(u):u\in H^{1}(M)\} = \inf \{\tilde{Q}(v):v\in H^{1/2}(\partial M) \} =: \tilde{Q}(M,\partial M).
    \end{align*}
\end{lemma}
\begin{proof}
    First note that $E(H^{1/2}(\partial M))\subset H^1(M)$, so we have:
    \begin{align*}
        \tilde{Q}(M,\partial M) 
        &= \inf \{Q(Ev):v\in H^{1/2}(\partial M) \}
        \geq \inf \{Q(u):u\in H^{1}(M)\} = Q(M,\partial M)
    \end{align*}
     For the other inequality, 
     we may use that a minimizer $u_*$ for $Q(M,\partial M)$ exists and solves \eqref{YPB}. In particular, $u_*= E \tr(u_*)$ and
     \[
Q(M,\partial M)= Q(u_*) = \widetilde{Q}(\tr u_*) \ge \tilde Q (M,\partial M). \qedhere
     \]
\end{proof}

In order to study $\tilde Q$ on $\cB$, we define for any $v \in \cB$ the tangent space
   \begin{align*}
        T_{v}\mathcal{B} = \left\{ u \in H^{1/2}(\partial M) : \int_{\partial M} v^{\frac{n}{n-2}}u d\sigma_{g} = 0\right\}
    \end{align*}
  and the associated orthogonal projector  
\begin{align*}\pi_{T_{v}\mathcal{B}} :H^{1/2}(\partial M) \mapsto T_{v}\mathcal{B} \subset H^{1/2}(\partial M) 
\quad \text{ by } \quad
    \pi_{T_{v}\mathcal{B}} \phi = \phi - \inprod{ v^{\frac{n}{n-2}}}{\phi} v .
\end{align*}
We start by computing the second variation of $\tilde Q$ on $\cB$. Throughout the paper, we will write $\nabla_{\cB}\tilde Q(u) = \nabla \tilde Q (u)\circ \pi_{T_{u}\cB}$ and similarly for $\nabla^2 Q_\cB$.
\begin{lemma} \label{lem:derivatives}
   
    For every $u \in \mathcal{B}$ the second variation of $\tilde Q$ on $\mathcal{B}$ is given by:
    \begin{align*}
     \frac{1}{2} \nabla_{\mathcal{B}}^{2} \tilde{Q} (u)[\varphi,\eta] 
     & = \inprod{\pi_{T_{u}\mathcal{B}}\varphi}{(\Lambda + \hat{c}_n h_{g} ) \pi_{T_{u}\mathcal{B}}\eta}
     - 
     \frac{n}{n-2} \tilde{Q}(u)\int_{\partial M}u^{\frac{2}{n-2}} (\pi_{T_{u}\mathcal{B}}\varphi ) (\pi_{T_{u}\mathcal{B}}\eta) d\sigma_{g}
\end{align*}
for all $\eta, \varphi \in H^{1/2}(\partial M)$.

\end{lemma}
\begin{proof}

For $u \in \mathcal{B}$ and $\varphi \in H^{1/2}(\partial M)$, the first variation of $\tilde{Q}$ is given by:
 \begin{align*}
     \nabla \tilde{Q}(u)[\varphi]   =& \|u\|_{L^{2\frac{n-1}{n-2}}}^{-4}\left[\|u\|_{L^{2\frac{n-1}{n-2}}}^{2}\left(2\hat{c}_n \int_{\partial M } h_{g}u\varphi d\sigma_{g} + 2\int_{\partial M}(\Lambda u)\varphi d\sigma_{g} \right) - \right.\\
     &\left. 2\left(\hat{c}_n \int_{\partial M } h_{g}u^{2} d\sigma_{g} + \int_{\partial M}(\Lambda u)u d\sigma_{g}\right)\|u\|_{L^{2\frac{n-1}{n-2}}}^{-\frac{2}{n-2}} \int_{\partial M}u^{\frac{n}{n-2}} \varphi d\sigma_{g} \right]
 \end{align*}
The last term vanishes for $\phi \in T_u\cB$ so 
\begin{align} \label{eq:NablaQ}
    \nabla_{\mathcal{B}} \tilde{Q}(u)[\varphi] =2 \inprod{\varphi}{(\Lambda + \hat{c}_n h_{g} ) u}
\end{align}
Taking one more derivative gives the second variation restricted to the tangent space
\begin{align*}
    \nabla_{\mathcal{B}}^{2} \tilde{Q}(u)[\varphi,\eta] 
    &= 2\inprod{\varphi}{(\Lambda + \hat{c}_n h_{g} ) \eta}
     - 
     \frac{2 n}{n-2} \tilde{Q}(u)\int_{\partial M}u^{\frac{2}{n-2}} \varphi \eta d\sigma_{g} \qedhere
\end{align*}
\end{proof}

\subsection{Local analysis near a minimizer.}
We start by defining the space of normalized optimizers for $\tilde Q$, which are boundary traces of optimizers for the original functional. Let
\begin{equation}
\label{eq:def_solution-variety}
\cM_\cB = \{ u \in \cB | Eu  \in \cM \}.
\end{equation}
For later use, we also define the set of normalized critical points
\begin{equation}
\label{eq:def_crit_points}
\cC_\cB = \{ u \in \cB | \nabla \tilde Q(u) = 0  \}.
\end{equation}

Throughout this subsection, we fix $v\in \cM_\cB$. The analysis in this section is largely analogous to the local stability analysis in \cite{engelstein2022quantitativestabilityminimizingyamabe}. However, we provide sufficient details to check the mapping properties related to the less common spaces and non-local operator $\Lambda$ that appears here.

The behavior of the functional $\tilde Q$ in a neighborhood of $v$ in $\cB$ is encoded in its second variation $\nabla^2_\cB \tilde Q(v)$. It is the quadratic form associated with the operator
$$
\cL_v := \Lambda + \hat c_n h_g - \frac{n }{n-2} Q(M,\partial M) v^{\frac{2}{n-2}}.
$$
Since $v$ is a minimum, $\cL_v$ is positive semi-definite. We denote by $K \subset T_v\cB$ its kernel, which is of dimension $l <\infty$ (possibly $l=0$, for instance, when $Q(M,\partial M) \le 0$) and consists of smooth functions.
We denote by $\pi_K$ the orthogonal projection on $K$ and by $K_{1/2}^\perp$ its orthogonal in $T_v\cB$, both with respect to the $L^2$ inner product.

In order to state the following lemma, we parametrize a neighborhood of $v$ in $\cB$ by the tangent space $T_v\cB = K \times K_{1/2}^\perp$ so we define
$$
h : K \times K^\perp_{1/2} \mapsto \cB \quad \text{ as }
 \quad h(a,b) =\norm{v +a + b}^{-1}_{L^p}( v + a + b).
$$

\begin{lemma} (Lyapunov-Schmidt reduction). \label{lem:Lyapunov-Schmidt}
    There exists a sufficiently small neighborhood $W$ of $v$ in $\cB$ and a continuous map $P:u\in 
    W\mapsto P(u) \in W$ with the following properties.
    \begin{enumerate}
        \item \label{item:parametrization}$P(W)$ is a $C^1$ manifold of dimension $l:= \dim (K)$. More concretely, there exists $V\subset K$, a neighborhood of the origin, and a $C^1$ function $F: V \mapsto K^\perp_{1/2}$ such that $P(W) = \{h(a, F(a)) | a \in V\}$, with $F(0)= 0$, $\nabla F(0) = 0$.
        \item \label{item:critical_points}If $w \in W$ is a critical point of $\tilde Q$, then $w = P(w)$. 
        \item \label{item:analiticity} The function $q : V \mapsto \R$ defined by $q(a) = \tilde Q(h(a, F(a)))$ is real analytic. (for any Euclidean norm in the finite dimensional $V$)
        \item \label{item:gradients} For all $a \in V$ and $\phi \in T_v\cB$, we have for some $c_a >0$
        \begin{equation}
            \nabla q(a)[\pi_K \phi]  = c_a\nabla_\cB \tilde Q((h(a,F(a)))[\phi].
        \end{equation}    
        \item \label{item:quadr_stability}
        There exists $C_1(v)>0$ such that for all $u\in W$, 
        \begin{equation} \label{eq:quadratic-stability}
            \tilde Q(u) -\tilde Q (P(u)) \ge C_1(v) \norm{u - P(u)}_{H^{1/2}}^2.
        \end{equation} 
    \end{enumerate}
\end{lemma}
If $K$ is trivial, this lemma is not necessary, but one may check that everything is consistent upon defining $P(u) = v$ for all $u\in W$. If $K$ is non-trivial, the lemma selects an $l$-dimensional variety $P(W)$ on which quadratic stability may fail, and an infinite dimensional space where the stability is indeed quadratic due to point~(\ref{item:quadr_stability}). The analyticity of $q$ ensures that we can apply {\L}ojasiewicz's inequality to obtain stability within this critical manifold, and points (\ref{item:critical_points}) and (\ref{item:gradients}) shows that all critical points \emph{close} to $v$ are included in this variety and correspond to critical points of $q$. The proof of Lemma~\ref{lem:Lyapunov-Schmidt} is given in Appendix~\ref{app:LSproof}.

This lemma allows to write for any $u\in W$
\begin{align*}
    \tilde{Q}(u)- Q(M,\partial M) 
    &= \tilde{Q}(u) - \tilde{Q}(P(u)) + \tilde{Q}(P(u))-Q(M,\partial M) \\
    &\ge C_1(v) \norm{u - P(u)}_{H^{1/2}}^2 + \tilde Q(P(u)) - \tilde Q(v).
\end{align*}
There are now two possible cases. In the \emph{integrable} case (see precise definition below), for each $u$, $P(u)$ is a minimizer as well. Thus, $\tilde Q(P(u)) = \tilde Q(v)$ and the inequality~\eqref{eq:quadratic-stability} gives a distance to the set of minimizers.
In the non-integrable case, not any $P(u)$ is a minimizer, the stability follows from {\L}ojasiewicz's inequality for $q$, and is not necessarily quadratic.

\medskip
The definition of integrability for our setting closely mimics the usual one in
\cite{engelstein2022quantitativestabilityminimizingyamabe}. 
We use the notation $\cC_\cB$ for the critical points of $\tilde Q$ in $\cB$.
\begin{definition}[Integrability] \label{def:integrability}
    A function $v \in \cC_\cB$ is integrable if for all $\varphi \in \Ker (\nabla_{\mathcal{B}}^{2}\tilde{Q}(v))$ there exists a one-parameter family of funtions $(v_{t})_{t\in (-\delta,\delta)}$, with $v_{0} = v $, $\frac{\partial}{\partial t} \mid_{t=0} v_{t}=\varphi$, and $v_{t} \in \tr(\cC_\cB)$ for all $t$ sufficiently small.
\end{definition}

The following lemma shows that if $v$ is an integrable minimizer, the variety $P(W)$ coincides with the intersection $W\cap \cC_\cB$.

\begin{lemma} \label{lem:integrability}
   Assume that $v \in \tr(\mathcal{M})$. We use the notation introduced in Lemma~\ref{lem:Lyapunov-Schmidt}, taking a smaller neighborhood $W$ if necessary. Then $v$ is integrable if and only if $q$ is constant in a neighborhood of $0 \in K$. In particular, if $v \in \tr(\mathcal{M})$ is an integrable minimizer, then 
    \begin{align*}
        \cM_\cB \cap W = P(W).
    \end{align*}
\end{lemma}
\begin{proof}
    If $q$ is constant, for each $\eta \in K$ the map
    $t \mapsto \phi_t:= h(t\eta, F(t\eta))$ has the desired properties.
    Conversely, if $v$ is integrable,
    by Lemma~\ref{lem:Lyapunov-Schmidt} point~(\ref{item:critical_points}), for each direction $b \in K$ we may find a curve
    $\phi_t = h(a_t, F(a_t))$ such that $a_0 = 0$, $\partial_t a_t = b$ and $\partial_t q(a_t) = 0$ for all sufficiently small $t$. This shows that $q$ vanishes to infinite order at $0$ and hence, since $q$ is real analytic, $q$ is constant.
\end{proof}

\section{General manifolds: Proof of the main theorems.}
Here we put together the ingredients to prove Theorem \ref{thm:mtheor}. The proof proceeds by first considering the problem close to the set of minimizers and concluding by compactness.

\subsection{Local stability estimate.}
We start by local quantitative stability for the functional $\tilde Q$.
Given $\delta >0$ and $v_* \in \cM_\cB$, we define a $0$-homogeneous quantity that measures the distance to minimizers in a $\delta$-neighborhood of $v_*$
\begin{align*}
    d_{v_*,\delta}^{\partial M}(v) = \| v \|_{H^{1/2}}^{-1} \,\inf \left\{\|v-\tilde{v} \|_{H^{1/2}(\partial M)} \middle| \tilde{v}\in \cM_\cB \cap B_{H^{1/2}}(v_*,\delta)\right\}
 \end{align*}

\begin{lemma} \label{lem:local_stab_bdry}
Fix a minimizer $v_* \in \cM_\cB$. There exist $C_2= C_2(v_*)> 0$, an exponent $\gamma= \gamma(v_*) \ge 0$ and a radius $\delta (v_* )>0 $ such that for every $v \in \cB\cap B_{H^{1/2}}(v_*,\delta)$
         we have:
        \begin{align*}
            \tilde{Q}(v) -Q(M,\partial M) \geq C_2 d_{v_*,\delta}^{\partial M}(v))^{2+\gamma}.
        \end{align*}
     \end{lemma}
Before proving this Lemma, we show that the local stability for $\tilde Q$ indeed implies local stability for $Q$.
For $u \in H^1(M)$ and $u_* \in \cM$ with $\tr u_* \in \cM_\cB$, we define the analogous local distance
 \begin{align*}
    d_{u_*,\delta}^{ M}(u) = \| u \|_{H^{1}(M)}^{-1}\, {\inf \{\|u-\tilde{u} \|_{H^{1}(M)} : \tilde{u}\in \mathcal{M} \cap B_{H^1}(u, \delta), \tr(\tilde u) \in \cB \}}.
  \end{align*}

\begin{proposition}[Local stability estimate] \label{prop:local_stab}
    Let $(M,g)$ be a compact Riemannian manifold with boundary $\partial M$, and $u_* \in \mathcal{M}$ with $\tr u_* \in \cB$. There exist constants $\delta(u_*)>0$, $C_3(u_*)>0$ and $\gamma(u_*) \ge 0$ such that for all $u \in B_{H^1}(u_*,\delta)$ with $\tr(u) \in \cB$ 
    \begin{align*}
        Q(u) -Q(M,\partial M) \geq C_3 d_{u_*,\delta}^{ M}(u) ^{2 +\gamma}
    \end{align*}
    If $v$ is integrable or non-degenerate, we may take $\gamma =0$.
\end{proposition}
\begin{proof}
    First, we transfer the problem to the boundary by using the boundary functional and the extension operator $E$ defined in \eqref{eq:def_harmonic_extension}. Every $u \in H^{1}(M)$ can be written as
    \begin{align*}
        u = u_{0} + E v
    \end{align*}
    where $v = \tr(u)$, and $u_{0} \in H^{1}_0(M)$ satisfies
    \begin{align*}
        c_{n}\int_{M} \nabla u_{0} \cdot \nabla (E v) + \int_{M} R_{g} u_{0} Ev &= 0.
    \end{align*}
 
 Thus, we find
    \begin{align*}
        Q(u) &= \tilde{Q}(v) + \frac{\langle u_{0}, (-c_{n}\Delta + R_{g})u_{0}\rangle}{\|u\|_{L^{2\frac{n-1}{n-2}}(\partial M)}^{2}} 
     \end{align*}
     Recall that $Y = \tilde{Y}$ (Lemma~\ref{lem:infequa}) and the standard trace Sobolev inequality, 
 $$
\norm{\tr(u)-\tr(u_*)}_{H^{1/2}(\partial M)} \le C \norm{u - u_*}_{H^1(\partial M)}, 
$$
so for $\delta(u_*) >0$ sufficiently small, we can apply Lemma~\ref{lem:local_stab_bdry} around $v_*:= \tr u_*$ to  estimate
     \begin{align*}
             Q(u) -Q(M,\partial M) &\geq C_2(v_*) d_{v_*,\delta}^{\partial M}( \tr u))^{2+\gamma} + \frac{\langle u_{0}, (-c_{n}\Delta + R_{g})u_{0}\rangle}{\|u\|_{L^{2\frac{n-1}{n-2}}(\partial M)}^{2}}.
             \end{align*}

Using Poincaré's inequality for the second term, the elementary inequality $a^{2+\gamma} +b^2 \ge C (a^2 + b^2)^{1+\gamma/2}$ valid for $0 \le a,b \le C$ and the Sobolev embedding gives
        \begin{align*}
   Q(u) -Q(M,\partial M)              &\geq C_2(v_*) d_{\delta}^{\partial M} (u,\tr(\mathcal{M}_{1}))^{2+\gamma} + c_{2}\frac{\|\varphi_{0}\|_{H^{1}(M)}^{2}  }{\|u\|_{H^{1}(M)}^{2}}\\
        &\gtrsim C_2(v_*) \left(  \frac{\inf \{ \|u - \tilde u \|_{H^{1}(M)} : v \in \mathcal{M}_{1}  \}}{\|\varphi\|_{H^{1}(M)} + \|u_{0}\|_{H^{1}(M)}}  \right)^{2+\gamma}\\
        &\gtrsim C_2(v_*)  d_{u_*,\delta}^{ M}(u)^{2+\gamma}
        \qedhere
    \end{align*}
\end{proof}

We now prove the quantitative stability near a minimizer of $\tilde Q$.

\begin{proof}[Proof of Lemma~\ref{lem:local_stab_bdry}]
We take $u \in \cB$ and fix $v \in \cM_\cB$ which will play the role of $v_*$. With the notations of Lemma~\ref{lem:Lyapunov-Schmidt} we take $\delta >0$ small enough so that $v \in W$. Throughout the proof, $C(v)$ will denote a positive constant that may depend on $v$ and the data of the problem, but not on $u$. We write as before, from Lemma~\ref{lem:Lyapunov-Schmidt} (~\ref{item:quadr_stability}) 
\begin{align} \label{eq:separation_LS}
  \tilde{Q}(u)- Q(M,\partial M)  
      &\ge C(v) \norm{u - P(u)}_{H^{1/2}}^2 + \tilde Q(P(u)) - \tilde Q(v).
\end{align}

The second term is trivial to estimate if $v$ is \textbf{nondegenerate} ( the case $l=0$) or $v$ \textbf{integrable}. In the first case, we have $P(u) = v$, in the second case from Lemma~\ref{lem:integrability} we conclude that $\tilde Q(P(u)) =\tilde Q(v)$ and $P(u)\in W \cap \cM_\cB$ . In either case, we have
\begin{align*}
    \tilde{Q}(u)- Q(M,\partial M)  
          \ge C_1(v) d^{\partial M}_{v,\delta} (u)^2 \norm{u}_{H^{1/2}}^2 \ge C_2(v) d^{\partial M}_{v,\delta} (u)^2 \norm{u}_{L^p}^2 
           = d^{\partial M}_{v,\delta} (u)^2 
\end{align*}

In the case that 
$v$ is \textbf{nonintegrable}, recall that $P(u) = h(a,F(a))$ for some $a\in V\subset K$.
We use the {\L}ojasiewicz inequality (recall that $q$ is analytic) and obtain $C_{L} > 0$ and  $\alpha > 0$ and $a' \in q^{-1}(\{q(v)\})$ such that
\begin{align*}
    \tilde Q(P(u)) - \tilde Q(v) &= q(a) - q(0)\\
    &\ge C_{L}\norm{a-a'}_K^{\alpha}.
\end{align*}
Here, the distance is with respect to any Euclidean norm in $V \subset K$. 
But
\[
q^{-1}(\{q(v)\}) = \{a \in V | h(a,F(a)) \in \cM_\cB \cap W \}.
\]
Furthermore, since $\nabla F(0)= 0$, and $h$ is to leading order an isometry, up to taking a smaller neighborhood $V$ we find for any $a,a'\in V$
\begin{align*}
    \norm{a -a'}_K \ge C(v)  \norm{a -a'}_{H^{1/2}} 
    &\ge C(v) \norm{(a,F(a)) -(a',F(a'))}_{H^{1/2}} \\
    &\ge C(v) \norm{h(a,F(a)) -h(a',F(a'))}_{H^{1/2}},
\end{align*}
so, with $v':= h(a', F(a'))$
\begin{align*}
   \tilde Q(P(u)) - \tilde Q(v)\ge C \norm{P(u) - P( v')}_{H^{1/2}}^\alpha.
\end{align*}
Having established this, with $2+\gamma = \max\{2,\alpha\}$, we return to \eqref{eq:separation_LS} and write
\begin{align*}
      \tilde{Q}(u)- Q(M,\partial M)   
      &\ge C(v) \norm{u - P(u)}_{H^{1/2}}^2 + C \norm{P(u) - P(\tilde v)}_{H^{1/2}}^{\alpha}  \\
      &\ge C(v) \norm{u -\tilde v}_{H^{1/2}}^{2+\gamma} \\
      &\ge C(v) d^{\partial M}_{v,\delta} (u)^{2 +\gamma}.
\end{align*}
which is the desired conclusion.

\end{proof}

\subsection{Compactness and global stability}
In this section we finish the proof of the global stability, using our local result and some classic compactness arguments. This last ingredient can is contained in the results of \cite[Theorem 3.5, the case $a=0$, $b=1$]{araujo2004existence} or for dimensions $3$ to $6$ in \cite{kim2021compactness}. Since these authors are concerned with stronger norms, their results require heavier machinery. We give the elementary proof for our case in Appendix~\ref{app:compactness}
 \begin{lemma}\cite{araujo2004existence,kim2021compactness}\label{lem:compactness}
   Let $(M,g)$ a compact Riemannian manifold with boundary of dimension $n \geq 3$ and such that $Q(M,\partial M) < Q(B^n,\partial B^{n})$. Let $(u_{k})_{k\in \mathbb{N}}$ be a minimizing sequence for $Q$, i.e. $Q(u_{k}) \to Q(M,\partial M)$. Then, up to a subsequence, there exists $v \in \mathcal{M} \cap \cB$ such that $u_{k} \norm{u_k}^{-1}_{L^p(\partial M)} \to v$ strongly in $H^{1}(M)$.
 \end{lemma}
We are now ready to prove the quantitative stability. We denote by $\cM_1$ the minimizers with normalized boundary trace, i.e., $\cM_1 =\{v \in \cM | \tr v \in \cB\}$.
\begin{proof}[Proof of Theorem~\ref{thm:mtheor}]
 Given $v \in \mathcal{M}_{1}$, by Proposition~\ref{prop:local_stab} there exist constants  $\delta(v), \gamma(v)$ and $C_3(v)$ for wich the local estimate holds. Since the set $\mathcal{M}_{1}$ is compact in $H^{1}( M)$, we can cover $\mathcal{M}_{1}$ by balls $\mathcal{B}_{H^{1}(M)}(v,\delta(v)/2)$ and take a finite subcover $\{\mathcal{B}_{H^{1}(M)}(v_{i},\delta(v_{i})/2)\}_{i\in I}$. Then we define:
    \begin{align*}
        \delta_{0} &:= \min_{i\in I} \delta(v_{i})/2 >0\\
        \gamma_{0} &:= \max_{i\in I} \gamma(v_{i}) <\infty \\
        c_{0} &:= \min_{i\in I} c(v_{i}) >0
    \end{align*}
Now argue by contradiction and assume that for each $n \in \N$, there exists $u_n \in H^1(M)$ such that
\[
Q(u_n) - Q(M,\partial M) \le n^{-1} \, d(u_n, \cM_1)^{2 + \gamma_0}.
\]
since both sides are homogeneous, we can consider $u_n \in \mathcal{B}$.
In this case, $\norm{u_n}_{H^1}$ is uniformly bounded from below and so $d(u_n, \cM_1)$ is bounded from above. So $(u_n)$ is a minimizing sequence and by Lemma~\ref{lem:compactness}, up to a subsequence, it converges strongly to some $v \in \cM_1$. In particular, there is $v_i$ in the finite subcover such that for all sufficiently large $n$, $\norm{v_i - u_n}_{H^{1/2}} < 2 \delta_0 \le \delta(v_i)$, in which case we obtain a contradiction with Lemma~\ref{lem:local_stab_bdry}. 
 
\end{proof}

We follow with the proof of the main corollary of Theorem \ref{thm:mtheor}, which discusses a conformally invariant version of the quantitative stability inequality obtained in this paper. The proof relies on the conformally invariance of both the conformal laplacian $L_{g} = -\frac{4(n-1)}{n-2}\Delta_{g} + R_{g}$ and of the boundary Neumman operator $B_{g} = \frac{4(n-1)}{n-2}\frac{\partial}{\partial \nu_{g}} + 2(n-1)h_{g}$. The proof closely follows the ideas in \cite{engelstein2022quantitativestabilityminimizingyamabe}, and we include it for the sake of completeness.

\begin{proof}[Proof of Corollary \ref{corollary}]

   Applying Theorem \ref{thm:mtheor} and the Sobolev trace inequality on $(M,g)$, if we plug $\tilde{g}=u^{\frac{4}{n-2}}g$ we obtain:
     \begin{align*}
         \mathcal{E}_{\tilde{g}} - Q(M,\partial M) = Q_{g}(u) - Q(M,\partial M) &\geq c \left(\frac{\inf_{v\in \mathcal{M}} \{\|u-v \|_{H^{1}(M)}:v \in \mathcal{M} \}}{\| u\|_{H^{1}(M)}}\right)^{2+\gamma}\\
         &\geq c \left(\frac{\inf_{v\in \mathcal{M}} \{\|u-v \|_{H^{1}(M)}:v \in \mathcal{M}\}}{\|u\|_{L^{2\frac{n-1}{n-2}}(\partial M)}}\right)^{2+\gamma} \\ &= c \left(\frac{\inf \{d_{[g]}^1(g,\tilde{g}) : \tilde{g}\in \mathcal{M}\}}{\text{vol}_{g}(\partial M)^{1/2_{*}}}\right)^{2+\gamma}
     \end{align*}
   The second inequality follows from the fact that if u is sufficiently close to the set of minimizers, i.e., $d(u, \mathcal{M}) \leq \delta_0$, the denominators are comparable. Otherwise, when $d(u, \mathcal{M})>\delta_0$, note that the quantity $\inf _{v \in \mathcal{M}}\|u-v\|_{H^{1}(M)} /\|u\|_{L^{\frac{2(n-1)}{n-2}}(\partial M)}$ is bounded above by 1. By choosing $c$ sufficiently small, the inequality holds. The proof of the subsequent inequality follows the same ideas.
\end{proof}
\subsection{Proof of Proposition \ref{prop:ASpint}}
Let $u_{0} \in \mathcal{M}$ be nonintegrable, and $q: U \to \mathbb{R}$ the perturbation function given by Lyapunov-Schmidt reduction, where $U \subset \text{ker}\nabla_{\mathcal{B}}^{2} \tilde{Q}(u_{0}) \cong \mathbb{R}^{l}$. Since $q$ is analytic, see \cite{Ottis,hoflow}, we can expand it in power series
\begin{align*}
    q(x) = q(0) + \sum_{j \geq p} q_{j}(x)
\end{align*}
where each $q_{j}$ is a degree $j$ homogeneous polynomial and $p$ is chosen so that $q_{p}(0)\neq 0$. As in \cite{hoflow}, we call this $p$ order of integrability of $u_{0}$. Next, we follow with the classic Adams-Simon positivity condition.
\begin{definition}[$\text{AS}_{p}$ condition]
    A minimizer $u_{0} \in \tr(\mathcal{M})$ satisfies $\text{AS}_{p}$ condition if it is nonintegrable and $q_{p}\mid_{\mathbb{S}^{l-1}}$ attains a positive maximum for some $\hat{v} \in \mathbb{S}^{l-1}\subset K$. Where $\mathbb{S}^{l-1}$ is the unit sphere with respect to the $L^{2\frac{n-1}{n-2}}(\partial M)$-norm in $K \subset H^{1/2}(\partial M)$. 
\end{definition}

As we did earlier in the local stability section, we prove a similar statement on the boundary and, after some manipulations, prove Proposition \ref{prop:ASpint}. The boundary lemma for this result is:
\begin{lemma}[Boundary $\text{AS}_{p}$ lemma]\label{lem:ASpboundary}
    Fix a compact Riemannian manifold with boundary of dimension $n \geq 3$, and $p \geq 3$. Let $v_{0}$ be a nonintegrable critical point of $\tilde{Q}$ and suppose it satisfies the $\text{AS}_{p}$ condition. Then there exists a sequence $(u_{i})_{i} \subset H^{1/2}(\partial M)$ with $u_{i}\to v_{0}$ in $H^{1/2}(\partial M)$, such that:
    \begin{align*}
        \lim_{i \to \infty} \frac{\tilde{Q}(u_{i})-Q(M,\partial M)}{\|u_{i} - v_{0} \|_{H^{1/2}(\partial M)}^{p-\alpha}} = 0 \quad \forall \alpha >0
    \end{align*}
\end{lemma}
\begin{proof}[Proof]
  We use the notation of Lemma~\ref{lem:Lyapunov-Schmidt}. 
    Let $\phi \in \mathbb{S}^{l-1}$ be the maximum of $q_{p}$, for $t$ sufficiently small, we define $u_t = h(t \phi, F(t\phi))$. 
     From the properties of $h$ and $F$ we see that
    \begin{align*}
        \|u_{t}-u_{0} \|_{H^{1/2}(\partial M)} = 
        \norm{h(t\phi, F(t\phi)) - h(0,0)} \ge C |t|
    \end{align*}
    The definition of $q$ yields
    \begin{align*}
        \tilde{Q}(u_{t}) - \tilde{Q}(v_{0}) = q(t \phi) - q(0) = \sum_{j \geq p}t^j q_{j}(\phi).
    \end{align*}
    With this we can establish for all sufficiently small $t$
    \begin{align*}
        |\tilde{Q}(u_{t}) - \tilde{Q}(u_{0})| 
        \le t^p q_p(\phi)/2.
    \end{align*}
   This implies the desired result thanks to the norm estimate obtained before.
\end{proof}
\par We now turn to the proof of Proposition \ref{prop:ASpint}, using arguments similar to those employed before.
\begin{proof}[Proof of Proposition~\ref{prop:ASpint}]
Apply $E$ to the functions $u_t$ defined in the proof of Lemma \ref{lem:ASpboundary}.
 
\end{proof}

Let us finish with some comments on the condition $AS_{p}$ in the case with boundary. In order to show the existence of a slowly converging Yamabe flow in manifolds with boundary, Ho and Shin in \cite[Section 5]{hoflow} provide an explicit example of a metric satisfying the $AS_{3}$ condition. For instance, if one takes $(M,g)$ to be a closed n-dimensional Riemannian manifold which is scalar-flat, i.e. $R_g \equiv 0$, and considers the $3$-dimensional unit ball $B^{3}$ equipped with the flat metric $g_0$, then the product manifold $M \times B^{3}$, equipped with the product metric
$g \oplus c^{-1}g_{0}$ is degenerate and satisfies $AS_{3}$. Here c is a constant chosen. A key ingredient to show it is that the Steklov eigenvalues of $B^{3}$ are well known.
Nevertheless, this is not an explicit example of a non-integrable minimizer of the Yamabe energy. 

An interesting problem in the boundary case would be to determine explicit examples of these nonintegrable minimizers satisfying $AS_{p}$, for $p\geq 4$. We expect bifurcations from the constant to occur for the family of product manifolds $M_T := \bbS^1(T) \times B^{n-1}(1)$ as the parameter $T$ increases. Since the equations involve a nonlinearity on the boundary, it is not obvious how to achieve such a result.

\appendix

\section{Details on the Lyapunov-Schmidt reduction} \label{app:LSproof}
In this section, we prove Lemma~\ref{lem:Lyapunov-Schmidt}. We work in a neighborhood of a fixed minimizer $v \in \tr \cM \cap \cB$. All functions are defined on $\partial M$. 
\begin{proof}[Proof of Lemma~\ref{lem:Lyapunov-Schmidt}.]

{\bf Banach spaces.}
Recall that $K \subset \cT_v \cB$ is the kernel of $\nabla^2_\cB Q(v)$. The elements of $K$ solve a linear elliptic equation, hence $K$ is a finite-dimensional subspace of $L^2$ and consists of smooth functions. We denote by $K^\perp$ its orthogonal in $L^2$ and by $K^\perp_{\pm 1/2}$ the closed subspaces of $H^{\pm 1/2}$ consisting of all $ \eta \in H^{\pm 1/2}$ such that $\langle \phi, \eta \rangle = 0  $ for all $\phi \in K$.
In anticipation of the inverse function theorem, we define
$$
h : K \times K^\perp_{1/2} \mapsto \cB \quad \text{ as }
 \quad h(a,b) =\norm{v +a + b}^{-1}_{L^p}( v + a + b).
$$ 
Since $\norm{v}_{L^p}=1$, $h$ is a homeomorphism between some small neighborhood $V\times U$ of the origin and a neighborhood $W$ of $v$ in $\cB$.
Next, we define $G(a,b) : K \times (K^\perp_{1/2}) \mapsto K^{\perp}_{-1/2}$ such that for all $\phi \in K^\perp_{1/2} $,
$ \langle G(a,b), \phi \rangle = \nabla_\cB \tilde Q (h(a,b)) [\phi] $. 

\noindent{\bf The implicit function Theorem.}
Note that $G(0,0) = 0$, since $v$ is a critical point.
We then compute its derivative with respect to the second variable by picking some $\eta \in K^\perp_{1/2}$ and computing
\begin{align*}
    \langle \partial_b G(0,0)[\eta] , \phi \rangle:=   \frac{d }{d t}\left.\langle G(0,t \eta), \phi \rangle \right|_{t=0} 
     &=   \frac{d }{d t}\left. \nabla_\cB \tilde Q \left(\frac{v + t\eta}{\norm{v + t\eta}_{L^p} }\right)[\phi] \right|_{t=0}  \\
     &= \nabla_\cB^2 \tilde Q \left( v \right)[\phi, \eta] \\
     & \ge \mu_1 \norm{\phi}_{H^{1/2}} \norm{\eta}_{H^{1/2}},
\end{align*}
where $\mu_1>0$. This shows that $\partial_b G(0,0)$ is an invertible operator from $K^\perp_{1/2} $ to $K^\perp_{-1/2}$.
By the implicit function theorem, up to fixing a smaller neighborhood $V\times U$ of the origin, the level set $\{(a,b) \in V\times U | G(a,b)= 0\}$ coincides with the graph of a function $F: V \mapsto K^\perp_{1/2}$. The function $F$ satisfies $F(0) = 0$, and 
$$
\nabla F(0) = - \partial_b G(0,0)^{-1} \partial_a G(0,0) = 0,
$$
where we have used that for all $\eta \in K$ and $\phi \in K^\perp_{1/2}$
$$
\langle \partial_a G(0,0)[\eta], \phi \rangle = \nabla_\cB^2 \tilde Q \left( v \right)[\phi, \eta] = 0.
$$
We finally set for $u = h(a,b) \in W$, $P(u) = h(a,F(a))$. This allows to define $P$ and proves point~\ref{item:parametrization}.

To check point~\ref{item:critical_points}, note that by definition, for any critical point $w=h(a,b)\in W$, $G(a,b) = 0$. Hence, $b = f(a)$ for all $(a,b)$ such that $h(a,b) \in W \cap \Crit_1$.

Since $K$ is finite dimensional, up to taking a smaller neighborhood if necessary, we find that
$$
\norm{a - a'}\ge c \norm{h(a,f(a)) - h(a',f(a'))}_{H^{1/2}}  
$$

\noindent{\bf Properties of $q$.}
We define the functional $q: V\subset K \to \R$ by $q(a) := \tilde Q(h(a,F(a)))$. Then we compute for any $\eta \in K$,
$$
\nabla q(a)[\eta] 
= c_a \nabla_\cB \tilde Q(h(a,F(a)))[\eta] + c_a \nabla_\cB \tilde Q(h(a,F(a)))[\nabla F [\eta]] = c_a \nabla_\cB \tilde Q(h(a,F(a)))[\eta]
$$
The number $c_a = \norm{v+ a+ F(a)}_p^{-1}$ is the only term that contributes from $h$, since the other term is perpandicular to $\cT_{h(a,F(a))} \cB$.
On the other hand, for an arbitrary $\phi \in \cT_v \cB$,
$$
\nabla_\cB \tilde Q(h(a,F(a)))[\phi] = \nabla_\cB \tilde Q(h(a,F(a)))[\pi_K \phi]
$$
so we derive the identity
$$
c_a \nabla q(a)[\pi_K \phi]  = \nabla_\cB Q(h(a,F(a)))[\phi],
$$
which settles point~\ref{item:gradients}.

\noindent{\bf Analiticity.}
Now we show that $q : K \mapsto \R$ is analytic in some neighborhood of the origin (point~\ref{item:analiticity}). We will use $h, F, G$ to refer to functions from the previous section restricted to smaller spaces of H\"older regular functions.
Note that $u \mapsto \norm{v+u}_{L^p}
$ is analytic from $C^0$ (or any smaller space) to $\R$ in a sufficiently small neighborhood of the origin since $v$ is bounded away from zero. This shows that $h$ is analytic. Next, 
from the explicit expression in \eqref{eq:NablaQ}, $G$ is real analytic from $C^{2,\alpha}\cap ( K \times K^\perp_{1/2})$ to $C^{1,\alpha} \cap K^\perp_{-1/2}$. 
Schauder estimates for the Neumann problem (e.g \cite[Theorem 6.30]{gilbarg2001elliptic}) then show that $\partial_b G(0,0)$ also has a continuous inverse between $C^{2,\alpha}\cap K^\perp_{1/2}$ and $C^{1
,\alpha}\cap K^\perp_{1/2}$.
By the implicit function theorem, $f$ is analytic from $K$ to $C^{2,\alpha}\cap K^\perp_{1/2}$. Composing $\tilde Q$, $h$ and $F$ finally gives the conclusion.

\noindent{\bf Coercivity.} It remains to establish point~\ref{item:quadr_stability}.
To this end, for some $u \in W$ we write $u= h(a, F(a) + \eta)$ and use Taylor's formula with remainder term to expand
\begin{align*}
    \tilde Q(u) -\tilde Q(P(u)) 
    &= c_a \nabla_\cB Q(h(a,F(a)))[\eta] + \frac{c_a}{2}\nabla_\cB^2(h(a,F(a)+\xi))[\eta, \eta] \\
    &= c_a \inprod{G(a, F(a))}{\eta} +  \frac{1}{2}\nabla_\cB^2(h(a,F(a)))[\eta, \eta] + \cR(\xi,\eta).
\end{align*}
The first term vanishes by the definition of $F$. 
For the final error term, we write $u_\xi = h(a, F(a)+\xi )$ and use the explicit expression of the derivatives in Lemma~\ref{lem:derivatives}
\begin{align*}
2 \cR(\xi,\eta): &=  \nabla_\cB^2(u_\xi)[\eta, \eta] - \nabla_\cB^2(P(u))[\eta, \eta] \\
&= \frac{n}{n-2} \int_{\partial M} |\eta|^2 \left(\tilde Q(u_\xi)u_\xi^{p-2} -\tilde Q(P(u))P(u)^{p-2}\right) \\
& \le C_{\cR}\norm{u_\xi -P(u)}_{H^{1/2}} \norm{\xi}_{H^{1/2}}^2.
\end{align*}
On the other hand, we have for a suitable constant $\mu_1$,
\begin{align*}
 \frac{1}{2}\nabla_\cB^2(h(a,F(a)))[\eta, \eta] 
 \ge \mu_1 \norm{\eta}^2_{H^{1/2}}.
\end{align*}
Thus, by taking a sufficiently small neighborhood $W$, we find for any $u \in W$
$$
C_{\cR}\norm{u_\xi -P(u)}_{H^{1/2}} \le \mu_1/2.
$$
This finally gives
\begin{align*}
    \tilde Q(u) -\tilde Q(P(u)) 
   &\ge  \mu_1/2 \norm{\eta}^2_{H^{1/2}} \\
   &\ge \mu_1/C \norm{u -P(u)}^2_{H^{1/2}}. \qedhere
\end{align*}

\end{proof}

\section{Compactness of the set of minimizers}
\label{app:compactness}
Our goal in this appendix is to prove the following statement.
\begin{proposition} \label{prop:M1_compact}
    If $(M,g)$ is a compact Riemannian manifold with boundary such that $Q(M,\partial M) < Q(B^n,\partial B^{n})$, then the set of minimizers (with normalized boundary trace) $\cM_1 $ is compact in $H^1$. 
\end{proposition}

Escobar \cite{escobar1992conformal} proves that the strict inequality implies the existence of a minimizer. The method he uses is to approximate the minimizer by optimizers of a variational principle with a subcritical exponent and show that one may take the limit. This does not directly imply the compactness.
 Since we could not find the simpler proof of compactness in $H^1$ in the literature, we give a direct proof of Proposition~\ref{prop:M1_compact}

\begin{proof}
We will show that any minimizing sequence $(u_n) \subset \cB$ has a subsequence that converges, in $H^1$, to a minimizer in $\cM_\cB$.
Define $\delta = Q(B^n,\partial B^{n}) - Q(M,\partial M) > 0$ and fix such a minimizing sequence, such that 
$$
Q(M,\partial M) = \lim_{n \to \infty} Q(u_n) = \lim_{n \to \infty} \int_M  |\nabla u_n|^2 + R_g u_n^2 .
$$
Since $u_n$ is bounded in $H^1$ and $L^p$, after extracting a (non-displayed) subsequence we can assume that it converges weakly in $H^1$ and $L^p$ to some $u$. We will use \cite{BrezisLiebConvergence} to show that this convergence is strong, which implies that $u \in \cB$ and that $u$ is a minimizer.
Writing $u_n = u + v_n $, we obtain by using the main result in \cite{BrezisLiebConvergence} and the simple bound $(a+b)^{p/2 }\ge a^{p/2} + b^{p/2}$ gives
\begin{align*}
    Q(M,\partial M)^{p/2}  = \lim_{n \to \infty} Q(u_n)^{p/2} 
    & = \lim_{n \to \infty} \left(\int_M  |\nabla u|^2 + R_g u^2 + \int_M  |\nabla v_n|^2 \right)^{p/2} \\
    &\le  \norm{u}_p^{p/2} Q(u)^{p/2} + \lim_{n \to \infty}  \norm{v_n}_p^{p/2} Q(v_n)^{p/2}.
\end{align*}
The curvature term in the definition of $Q(v_n)$ is irrelevant in the limit since the weak $H^1$ convergence implies strong convergence in $L^2(M)$.
Below, we will show that for any sequence converging to zero weakly,
\begin{equation} \label{eq:lower_bound_Qgn}
    \liminf_{n \to \infty} Q(v_n) \ge Q(B^n,\partial B^{n}) - \delta/2.
\end{equation}
This implies that $\lim_{n \to \infty}  \norm{v_n}_p = 0$, wich in turn implies that $\lim_{n \to \infty}  \norm{\nabla v_n}_2 = 0$ and concludes the proof.

In order to prove \eqref{eq:lower_bound_Qgn}, we use a localization argument.
Take some $\epsilon > 0$ sufficiently small.
Then, we can assume that $M_\epsilon := \{x \in M | \dist(x,\partial M) < \epsilon\}$ is isomorphic to $(\partial M \times [0,\epsilon), g_{\partial M} \times 1 + \cO(\epsilon))$.
Take a smooth cut-off functions $\chi, \phi$ such that $\chi^2  + \phi^2 = 1$, $\chi$ is supported in $M_\epsilon$ and $\phi$ in $\operatorname{int} (M)$, such that $\int_{\partial M} v_n^p = \int_{\partial M} (\chi v_n)^p $. For the gradient term,
by the IMS localization formula,
$$
\int_M  |\nabla v_n|^2 = \int_M  \left(|\nabla \chi v_n|^2 + |\nabla \phi v_n|^2 \right) - \int_M \left(|\nabla \phi|^2 + |\nabla \chi|^2  \right)v_n^2
$$
Dropping the interior term and using the strong $L^2$ convergence to eliminate the error term gives a first lower bound
$$
\liminf_{n \to \infty} Q(v_n) \ge \liminf_{n \to \infty}  \frac{\int_M |\nabla \chi v_n|^2}{\norm{\chi v_n}_{L^p(\partial M)}^{2}}.
$$
Now, we consider a further partition of unity to compare the variational quotient for functions supported close to the boundary to functions supported in small patches, where the metric can be taken close to the one for the half-space. 
For a given $\epsilon > 0$, we fix a partition of unity $\chi_j : M \mapsto [0,1]$ such that $\sum_j \chi_j^2 = 1$ for each $x \in M_\epsilon$. Furthermore, each $\phi_j$ is supported in a $B(x_j,2 \epsilon)$ for some $x_j \in \partial M$. By choosing coordinates adapted at the point $x_j$ in a standard way, we find
$$
\int_M |\nabla \phi_j v_n|^2 \ge ( Q(B^n,\partial B^{n})- \cO(\epsilon))  \norm{\phi_j v_n}_{L^p(\partial M)}^{2} - \cO(\epsilon) \norm{\phi_j v_n}_{L^2(M)}^2. 
$$
Using once more the IMS localization formula and the $L^2$ convergence, we obtain
\begin{align*}
  \liminf_{n\to \infty} \int_M |\nabla v_n|^2 
  & \ge \liminf_{n\to \infty} \sum_{j} \int_M |\nabla \phi_j v_n|^2 \\
   & \ge ( Q(B^n,\partial B^{n})- \cO(\epsilon)) \liminf_{n\to \infty} \sum_j \left(\int_{\partial M} \phi_j^p v_n^p \right)^{2/p}\\
   & \ge ( Q(B^n,\partial B^{n})- \cO(\epsilon)) \liminf_{n\to \infty}  \left(\sum_j\int_{\partial M} \phi_j^p v_n^p \right)^{2/p}\\
   & = ( Q(B^n,\partial B^{n})- \cO(\epsilon)) \liminf_{n\to \infty}  \left(\norm{v_n}_{L^p(\partial M)}^p -\int_M\sum_j (\phi_j^2 -\phi_j^p) v_n^p \right)^{2/p}.
\end{align*}
Note that $\sum_j (\phi_j^2 -\phi_j^p) \in [0,1)$ and that this quantity is supported in $T:= \cup_j \phi_j^{-1}((0,1)) \cap \partial M$. However, it is not necessarily a small quantity for an arbitrary partition of unity.
To solve this problem, a solution, detailed for instance in \cite[Lemma 2.1]{benguria2020existence} for the Euclidean case, is to consider a family of partitions of unity $\phi_{j,v}$ such that for each $v$, $|T(v)|=|\cup_j \phi_{j,v}^{-1}((0,1)) \cap \partial M| \le \epsilon |\partial M|$ and such that, when averaging over $v$ with a suitable probability measure $\di \mu (v)$, for each $x \in \partial M$,
$\int 1_{T(v)}(x) \di \mu(v) \le \cO(\epsilon)$.
This allows finally to bound
\begin{align*}
  \liminf_{n\to \infty} \int_M |\nabla v_n|^2 
    & \ge ( Q(B^n,\partial B^{n})- \cO(\epsilon)) \liminf_{n\to \infty} \int \left(1 -\int_M1_{T(v)} v_n^p \right)^{2/p}\di \mu(v) \\
     & \ge ( Q(B^n,\partial B^{n})- \cO(\epsilon)) \liminf_{n\to \infty}  \left(1 -\int \int_M 1_{T(v)}(x) v_n^p \di \mu(v)\right) \\
      & \ge ( Q(B^n,\partial B^{n})- \cO(\epsilon))  \liminf_{n\to \infty}\norm{v_n}_{L^p(\partial M)}. \qedhere
\end{align*}
\end{proof}

\end{document}